\documentclass[11pt]{amsart} \usepackage{latexsym, amssymb}
\usepackage[T1]{fontenc}

\newtheorem{thm}{Theorem}[section]
\newtheorem{lemma}[thm]{Lemma}
\newtheorem{prop}[thm]{Proposition}

\newtheorem{fact}[thm]{Fact}

\theoremstyle{definition}

\theoremstyle{remark}

\renewcommand{\r}{\mathbb{R}}
\newcommand{\Z}{\mathbb{Z}}

\newcommand{\n}{\mathbb{N}}

\renewcommand{\to}{\rightarrow}

\def \<{\langle}
\def \>{\rangle}

\def \*Z {{{^*}\Z}}
\def \((  {(\!(}
\def \)) {)\!)}

\numberwithin{equation}{section}

\def \u{\mathcal U}

\makeatletter

\def\indsym#1#2{%
  \setbox0=\hbox{$\m@th#1x$}%
  \kern\wd0%
  \hbox to 0pt{\hss$\m@th#1\mid$\hbox to 0pt{$\m@th#1^{#2}$}\hss}%
  \lower.9\ht0\hbox to 0pt{\hss$\m@th#1\smile$\hss}%
  \kern\wd0}

\def\nindsym#1#2{%
  \setbox0=\hbox{$\m@th#1x$}%
  \kern\wd0%
  \hbox to 0pt{\hss$\m@th#1\not$\kern1.4\wd0\hss}
  \hbox to 0pt{\hss$\m@th#1\mid$\hbox to 0pt{$\m@th#1^{\,#2}$}\hss}%
  \lower.9\ht0\hbox to 0pt{\hss$\m@th#1\smile$\hss}%
  \kern\wd0}

\allowdisplaybreaks[2]

\newcommand{\cstar}{$\mathrm{C}^*$}

\title{Boundary amenability of groups via ultrapowers}
\author{Stephen Avsec and Isaac Goldbring}
\thanks{Goldbring's work was partially supported by NSF CAREER grant DMS-1349399.}

\email{stephen.avsec@gmail.com}

\address {Department of Mathematics, University of California, Irvine, Irvine, CA, 92697-3875.}
\email{isaac@math.uci.edu}
\urladdr{http://www.math.uci.edu/~isaac}

\begin{document}

\begin{abstract}
We use \cstar-algebra ultrapowers to give a new construction of the Stone-Cech compactification of a separable, locally compact space.  We use this construction to give a new proof of the fact that groups that act isometrically, properly, and transitively on trees are boundary amenable.
\end{abstract}

\maketitle

\section{Introduction}\label{amen}

Suppose that the discrete group $\Gamma$ acts continuously on a compact space $X$.  We say that the action of $\Gamma$ on $X$ is \emph{amenable} if there is a net of continuous functions $x\mapsto \mu_n^x:X\to P(\Gamma)$ such that, for all $\gamma\in \Gamma$, we have $$\sup_{x\in X}\|\gamma \cdot \mu_n^x-\mu_n^{\gamma\cdot  x}\|_1\to 0.$$  We say that $\Gamma$ is \emph{boundary amenable} if $\Gamma$ acts amenably on some compact space.  Note that amenable groups are precisely the groups that act amenably on a one-point space, whence they are boundary amenable.  A prototypical example of a boundary amenable group that is not amenable is any non-abelian finitely generated free group.  Boundary amenable groups are sometimes referred to as exact groups for the reduced group \cstar-algebra $\mathrm{C}^*_r(\Gamma)$ is exact (meaning that the functor $\otimes_{\min}\mathrm{C}^*_r(\Gamma)$ is exact) if and only if $\Gamma$ is boundary amenable.

In this note, we show how one can construct the Stone-Cech compactification of a separable, locally compact space using \cstar-algebra ultrapowers.  When applied to the case of a tree, this construction gives a very natural proof of the fact that a group that acts isometrically, properly, and transitively on a tree is boundary amenable.  It was our initial hope that this construction could be used to settle the boundary amenability of groups where the answer was unknown (most notably \emph{Thompson's group}) but we have thus far been unsuccessful (although remain optimistic).  The na\"ive idea behind our optimism is that groups such as Thompson's group ``almost'' act isometrically on a tree and it is often the case that ultrapower constructions can turn almost phenomena into exact ones.

In Section 2, we explain the needed background on groups acting on \cstar-algebras as well as ultrapowers of \cstar-algebras.  In Section 3, we explain our main construction in the general setting of separable, locally compact spaces.  Finally, in Section 4 we use our construction to prove the boundary ameanbility of groups acting isometrically, properly, and transitively on trees.

\section{Preliminaries}

\subsection{Boundary amenability of groups acting on \cstar-algebras}

We will verify that certain groups act amenably on a compact space by checking that the group acts amenably on a unital abelian \cstar-algebra as we now explain.  Suppose that $B$ is a unital \cstar-algebra and that $\Gamma$ acts on $B$.
We consider the space $C_c(\Gamma,B)$ of finitely supported functions $\Gamma\to B$.  $C_c(\Gamma,B)$ is naturally a $*$-algebra with respect to the convolution product
$$(f*g)(\gamma)=\sum_{\gamma_1\cdot \gamma_2=\gamma}f(\gamma_1)(\gamma_1\cdot g(\gamma_2))$$ and involution
$$f^*(\gamma)=\gamma\cdot f(\gamma^{-1})^*.$$  We also view $C_c(\Gamma,B)$ as a pre-Hilbert $B$-module with $B$-valued inner product $\langle f,g\rangle_B=\sum_{\gamma \in \Gamma}f(\gamma)^*g(\gamma)$ and corresponding norm $\|f\|_{B}:=\|\langle f,f\rangle_B\|^{-1/2}$.

Recall also that an action of $\Gamma$ on a compact space $X$ induces an action of $\Gamma$ on $C(X)$ by $(\gamma \cdot f)(x):=f(\gamma^{-1}x)$.

Our approach to showing that groups are boundary amenable is via the following reformulation of amenable actions (see \cite[Proposition 2.2]{kerr}).

\begin{fact}
The action $G\curvearrowright X$ is amenable if and only if there exists a net $T_i\in C_c(G,C(X))$ such that, for each $\gamma\in \Gamma$ and each $i$, we have:
\begin{enumerate}
\item $T_i(\gamma)\geq 0$;
\item $\langle T_i,T_i\rangle_{C(X)}=1$;
\item $\|T_i-\delta_\gamma*T_i\|_{C(X)}\to 0$.
\end{enumerate}
\end{fact}

In our proofs below, we will have an action of a group $\Gamma$ on a unital, abelian \cstar-algebra $B$ and we will prove that there exist $T_i\in C_c(\Gamma,B)$ satisfying the clauses (1)-(3) in the aforementioned fact.  By Gelfand theory, $B$ is isomorphic to $C(X)$ for some compact space $X$.  It remains to observe that Gelfand theory respects the group action, meaning that we obtain an induced action of $\Gamma$ on $X$ such that the corresponding action of $\Gamma$ on $C(X)$ ``is'' the corresponding action of $\Gamma$ on $B$.  Thus, our criterion for boundary amenability of a group is the following:

\begin{fact}\label{test}
A group $\Gamma$ is boundary amenable if and only if there is a unital, abelian \cstar-algebra $B$ and a net $T_i\in C_c(G,B)$ such that, for each $\gamma\in \Gamma$ and each $i$, we have:
\begin{enumerate}
\item $T_i(\gamma)\geq 0$;
\item $\langle T_i,T_i\rangle_B=1$;
\item $\|T_i-\delta_\gamma*T_i\|_{B}\to 0$.
\end{enumerate}
\end{fact}

\subsection{Ultrapowers of $C^*$ algebras}\label{ultra}

Recall that a nonprincipal ultrafilter $\u$ on $\n$ is a $\{0,1\}$-valued measure on \emph{all} subsets of $\n$ such that finite sets get measure $0$.  We usually identify a nonprincipal ultrafilter with its collection of measure 1 sets, whence we write $A\in\u$ to indicate that $A$ has measure $1$.  If $P(n)$ is a property of natural numbers, we say that $P(n)$ holds $\u$-almost everywhere if the set of $n$ for which $P(n)$ holds belongs to $\u$.  If $(r_n)$ is a bounded sequence of real numbers, then the \emph{ultralimit of $(r_n)$ with respect to $\u$}, denoted $\lim_{n,\u} r_n$ or even $\lim_\u r_n$, is the unique real number $r$ such that, for every $\epsilon>0$, we have $|r_n-r|<\epsilon$ $\u$-almost everywhere.

Suppose that $A$ is a unital \cstar-algebra and $\u$ is a nonprincial ultrafilter on $\n$.  We can define a seminorm $\|\cdot \|_\u$ on $\ell^\infty(A)$ by setting $\|(f_n)\|_\u:=\lim_\u \|f_n\|$.  We set $A^\u$ to be the quotient of $\ell^\infty(A)$ by those elements of $\|\cdot\|_\u$-norm $0$; we refer to $A^\u$ as the \emph{ultrapower of $A$ with respect to $\u$}.  It is well known that $A^\u$ is once again a unital \cstar-algebra.  For $(f_n)\in \ell^\infty(A)$, we let $(f_n)^\bullet$ denote its image in $A^\u$.  The canonical \emph{diagonal embedding} $\Delta:A\to A^\u$ is given by $\Delta(a)=(a)^\bullet$.

\section{The main construction}

In this section, we consider a second countable, locally compact space $X$ with fixed basepoint $o\in X$.  It is well-known that $X$ admits a compatible proper metric $d$ (see \cite[Theorem 2]{V}), and we fix such a metric in the rest of this section.  For $r\in \r^{>0}$, we set $B(r)$ to be the closed ball of radius $r$ around $o$.

We set $A=C_o(X)$, the space of complex-valued continuous functions on $X$ that vanish at infinity.  For $(f_n)\in \ell^\infty(A)$, we say that $(f_n)$ is \emph{$\u$-equicontinuous on bounded sets} if, for every $r,\epsilon>0$, there is $\delta>0$ such that, for $\u$-many $n$, we have for all $s,t\in B(r)$ with $d(s,t)<\delta$, that $|f_n(s)-f_n(t)|\leq \epsilon$.

Given any $(f_n)\in \ell^\infty(A)$, set $f_\u:X\to \mathbb{C}$ by $f_\u(t):=\lim_\u f_n(t)$.  Note that $f_\u$ is a bounded function.  The following lemma is quite routine and left to the reader.

\begin{lemma}
If $(f_n)$ is $\u$-equicontinuous on bounded sets, then $f_\u$ is uniformly continuous on bounded sets.
\end{lemma}

\begin{lemma}
Suppose that $(f_n)^\bullet=(g_n)^\bullet$ and $(f_n)$ is $\u$-equicontinuous on bounded sets.  Then so is $(g_n)$.
\end{lemma}

\begin{proof}
Fix $r,\epsilon>0$.  Take $\delta>0$ that witnesses $\u$-equicontinuity of $(f_n)$ on $B(r)$ for $\epsilon/3$.  Then for $\u$-many $n$, we have, for $s,t\in B(r)$ with $d(s,t)<\delta$, that
$$|g_n(s)-g_n(t)|\leq 2\|g_n-f_n\|+|f_n(s)-f_n(t)|\leq\frac{2\epsilon}{3}+\frac{\epsilon}{3}=\epsilon.$$
\end{proof}

\noindent The previous lemma allows us to consider the \emph{continuous part of $A^\u$}
$$A^{c\u}:=\{(f_n)^\bullet \in A^\u \ : \ (f_n) \text{ is $\u$-equicontinuous on bounded sets}\}.$$

\begin{lemma}

\

\begin{enumerate}
\item $A^{c\u}$ is a \cstar-subalgebra of $A^\u$.  
\item $\Delta(A)\subseteq A^{c\u}$.
\end{enumerate}
\end{lemma}

\begin{proof}
For (1), it is clear that $A^{c\u}$ is a $*$-subalgebra of $A^\u$.  We must show that $A^{c\u}$ is closed in $A^\u$.  Towards this end, suppose that $(f_n^m)^\bullet$ is a sequence in $A^{c\u}$ such that $(f_n^m)^\bullet\to (h_n)^\bullet$ as $n\to \infty$.  We need to show that $(h_n)^\bullet\in A^{c\u}$.  Fix $r,\epsilon>0$.  Choose $m\in \n$ such that $\|(f_n^m)^\bullet-(h_n)^\bullet\|<\epsilon/3$.  Suppose that $s,t\in B(r)$ are such that $d(s,t)<\delta$.  Then we have that, for $\u$-many $n$, that
$$|h_n(s)-h_n(t)|\leq 2\|f_n^m-h_n\|+|f_n^m(s)-f_n^m(t)|\leq \frac{2\epsilon}{3}+\frac{\epsilon}{3}=\epsilon.$$

(2) follows from the fact that balls $B(r)$ in $X$ are compact, whence elements of $A$ are uniformly continuous on such balls.
\end{proof}

\noindent We now consider
$$I:=\{(f_n)^\bullet \in A^{\u} \ : \ (\exists r_n\in \r) (\lim_\u r_n=+\infty \text{ and } f_n|B(o,r_n)\equiv 0)\}.$$
It is clear from the definition that $I\subseteq A^{c\u}$.


In the rest of this section, we fix continuous functions $\chi_n:X\to \r$ such that:
\begin{enumerate}
\item $0\leq \chi_n\leq 1$;
\item $\chi_n(t)=0$ for $t\in B(n)$;
\item $\chi_n(t)=1$ when $d(t,o)\geq n+1$.
\end{enumerate}
\begin{prop}

\

\begin{enumerate}
\item $I$ is a closed ideal in $A^{\u}$.
\item $A^{c\u}/I$ is unital.
\item $q\circ \Delta:A\to A^{c\u}/I$ is injective, where $q:A^{c\u}\to A^{c\u}/I$ is the canonical quotient map.
\item $(q\circ \Delta)(A)$ is an essential ideal in $A^{c\u}/I$.
\end{enumerate}
\end{prop}

\begin{proof}
For (1), suppose $(f_n)^\bullet,(g_n)^\bullet\in I$, $(h_n)\in A^\u$, and $\lambda\in \mathbb C$.  Suppose that $f_n|B(r_n),g_n|B(s_n)\equiv 0$, where $\lim_\u r_n=\lim_\u s_n=0$.  Then $$\lambda f_n|B(r_n),(f_n+g_n)|B(\min(r_n,s_n)),(f_n\cdot h_n)|B(r_n)\equiv 0;$$ since $\lim_\u \min(r_n,s_n)=\infty$, we have $\lambda f_n,f_n+g_n,f_n\cdot h_n\in I$ and $I$ is an ideal.

We now prove that $I$ is closed.  Suppose that $((f_n^m)^\bullet \ : \ m\in \n)$ is a sequence from $I$ such that $\lim_m (f_n^m)^\bullet=(g_n)^\bullet$; we must show that $(g_n)^\bullet\in I$.  Suppose that $f_n^m|B(r_n^m)\equiv 0$ with $\lim_{n,\u}r_n^m=\infty$ for each $m$.  Fix $k\in \n$ and take $m\in \n$ such that $\|(f_n^m)^\bullet-(g_n)^\bullet\|<\frac{1}{k}$.  For $\u$-many $n$ we have $\|f_n^m-g_n\|<\frac{1}{k}$ and $r_n^m\geq k$.  Thus, if we set $$X_k:=\{n\in \n \ : n\geq k \text{ and } |g_n(t)|<\frac{1}{k} \text{ for }t\in B(k)\},$$ we have that $X_k\in \u$.  For $n\in \n$, set $l(n):=\max\{k\in \n \ : \ n\in X_k\}$.  Note that $n\in X_k$ implies that $l(n)\geq k$, whence $\lim_{n,\u}l(n)=\infty$.  Define $h_n:=f_n\cdot \chi_{l(n)-1}$.  Note that $(h_n)^\bullet \in I$; it remains to show that $(g_n)^\bullet=(h_n)^\bullet$.  For $n\in \n$, we have $\|g_n-h_n\|\leq \sup_{t\in B(l(n))}|g_n(t)|\leq \frac{1}{l(n)}$, whence
$$\|(g_n)^\bullet-(h_n)^\bullet\|=\lim_\u \|g_n-h_n\|\leq \lim_\u \frac{1}{l(n)}=0.$$

For (2), consider any sequence $(g_n)\in \ell^\infty(A)$ such that $g_n\equiv 1$ on $B(n)$.  (For example, take $g_n:=1-\chi_n$.)  We claim that $q(g_n)^\bullet$ is an identity for the larger algebra $A^{\u}/I$.  Indeed, consider arbitrary $q(f_n)^\bullet\in A^\u/I$.  Then $f_ng_n-f_n$ vanishes on $B(n)$, whence $(f_ng_n-f_n)^\bullet\in I$ and $q(f_ng_n)^\bullet=q(f_n)^\bullet$.

For (3), suppose that $(q\circ \Delta)(f)=0$.  Then there is $(g_n)^\bullet \in I$ such that $\Delta(f)=(g_n)^\bullet$.  Suppose that $g_n|B(r_n)\equiv 0$ with $\lim_\u r_n=\infty$.  Fix $t\in X$ and $\epsilon>0$.  Then for $\u$-many $n$, we have $\|f-g_n\|<\epsilon$ and $t\in B(r_n)$, whence $|f(t)|<\epsilon$.  Since $t$ and $\epsilon$ were arbitrary, we have that $f\equiv 0$.

We now prove (4).  We first show that $(q\circ \Delta)(A)$ is an ideal in $A^{c\u}/I$.  Towards this end, fix $f\in A$ and $q((g_n)^\bullet)\in A^{c\u}/I$; we must show that $q((fg_n)^\bullet) \in q(\Delta(A))$.  In fact, we will show that $q((fg_n)^\bullet)=q(\Delta(fg_\u))$.  Recall that $$\|q((fg_n)^\bullet)-q(\Delta(fg_\u))\|=\inf\{\lim_{\u} \|fg_n-fg_\u-h_n\| \ : \ (h_n)^\bullet \in I\}.$$  Set $M:=\sup_n \|g_n\|$.  Fix $\epsilon>0$.  Fix $m\in \n$ such that $|f(t)|<\frac{\epsilon}{2M}$ when $t\in B(m)^c$.  Let $\delta>0$ witness the $\u$-equicontinuity of $(g_n)$ on $B(m)$ with respect to $\frac{\epsilon}{3\|f\|}$ and fix a finite $\delta$-net $\{t_1,\ldots,t_k\}$ for $B(m)$.  Fix $U\in \u$ such that $\{k\in \n \ : \ k\geq m\}\subseteq U$ and $|g_n(t_i)-g_\u(t_i)|<\frac{\epsilon}{3\|f\|}$ for $i=1,\ldots,k$ and $n\in U$.  For $n\in U$, define $h_n\in A$ by $h_n:=(fg_n-fg_\u)\chi_n$.  (Define $h_n\in A$ for $n\notin U$ in an arbitrary fashion).  It suffices to show that $\lim_\u \|fg_n-fg_\u-h_n\|\leq\epsilon$.  Suppose $n\in U$.  First consider $t\in B(m)$.  Then $|fg_n(t)-fg_\u(t)-h_n(t)|=|fg_n(t)-fg_\u(t)|$.  Take $i$ such that $d(t,t_i)<\delta$.  Then, for $\u$-many $n$, we have $$|g_n(t)-g_\u(t)|\leq |g_n(t)-g_n(t_i)|+|g_n(t_i)-g_\u(t_i)|+|g_\u(t_i)-g_\u(t)|\leq\frac{\epsilon}{\|f\|},$$ whence $|fg_n(t)-fg_\u(t)|\leq\epsilon$.  Now suppose that $t\in B(m)^c\cap B(n+1)$.  Then $|fg_n(t)-fg_\u(t)-h_n(t)|\leq |fg_n(t)-fg_\u(t)|<\epsilon$ by choice of $m$.  If $t\in B(n+1)^c$, then $fg_n(t)-fg_\u(t)-h_n(t)=0$.  It follows that $\lim_\u\|fg_n-fg_\u-h_n\|\leq\epsilon$, finishing the proof that $(q\circ\Delta)(A)$ is an ideal in $A^{c\u}$.

We next show that $(q\circ \Delta)(A)$ is an essential ideal in $A^{c\u}/I$.  Suppose that $q(f_n)^\bullet\in A^{c\u}/I$ is such that $q(f_n)^\bullet\cdot q(a)^\bullet=0$ for all $a\in A$; we must show that $q(f_n)^\bullet=0$.

Fix $t\in X$.  Fix $a\in A$ such that $a(t)=1$.  Then there is $(g_n)^\bullet\in I$ such that $\lim_\u \|f_na-g_n\|=0$.  For $\u$-most $n$, we have $t\in B(r_n)$, where $g_n$ vanishes on $B(r_n)$.  It thus follows that
$$\lim_\u |f_n(t)|\leq \lim_\u \|f_na-g_n\|=0.$$

Set $$U_k:=\{n\in \n \ : \ n\geq k \text{ and } |f_n(t)|\leq \frac{1}{k} \text{ for} \ t\in B(k)\}.$$
We claim that $U_k\in \u$. Fix $\delta>0$ that witnesses $\u$-equicontinuity of $(f_n)$ on $B(k)$ with respect to $\frac{1}{2k}$.  Fix a finite $\delta$-net $F$ for $B(k)$.  Then for $\u$-most $n$, $|f_n(t)|\leq \frac{1}{2k}$ for $t\in F$.  Thus, given any $s\in B(k)$ and taking $t\in F$ such that $d(s,t)<\delta$, we have that $|f_n(s)|\leq |f_n(s)-f_n(t)|+|f_n(t)|\leq\frac{1}{k}$ for $\u$-most $n$. 

For $n\in \n$, set $l(n):=\max \{k\in \n \ : \ n\in U_k\}$.  For $n\in U_k$, we have $l(n)\geq k$, whence $\lim_\u l(n)=\infty$.  Define $h_n\in A$ by $h_n=f_n\cdot \chi_{l(n)-1}$.  As above, we have that $(h_n)^\bullet \in I$ and $\|f_n-g_n\|\leq \frac{1}{l(n)}$ whence $\lim_\u \|f_n-g_n\|\leq \lim_\u \frac{1}{l(n)}=0$.
\end{proof}

%
%
%

Since $q(\Delta(A))$ is an essential ideal in the unital \cstar-algebra $A^{c\u}/I$, we see that $\Sigma(A^{c\u}/I)$ is a compactification of $X$, where $\Sigma(A^{c\u}/I)$ denotes the Gelfand spectrum of $A^{c\u}/I$.  It turns out that this compactification is indeed the Stone-Cech compactification of $X$.  Recall that $C_b(X)$ denotes the unital \cstar-algebra of bounded, continuous, complex-valued functions on $X$ and is natrually isomorphic to $C(\beta X)$, where $\beta X$ denotes the Stone-Cech compactification of $X$.

\begin{prop}
There is an isomorphism $\Phi:A^{c\u}/I\to C_b(X)$ such that $\Phi(q(\Delta(a))=a$ for all $a\in A$.
\end{prop}

\begin{proof}
Define $\Phi:A^{c\u} \to C_b(X)$ by $\Phi((f_n)^\bullet):=f_\u$.  It is clear that $\Phi$ is a *-morphism.  We next observe that $\Phi$ is onto.  Indeed, given $f\in C_b(X)$ and $n>0$, define $f_n\in C_0(T)$ by $f_n=(1-\chi_n)f$.  Since $f$ is bounded, we have that $(f_n)\in \ell^\infty(A)$.  Since $X$ is proper, $f$ is uniformly continuous on bounded sets, whence $(f_n)$ is $\u$-equicontinuous on bounded sets, that is, $(f_n)^\bullet\in A^{c\u}$.  It is clear that $\Phi((f_n)^\bullet)=f$.

Now suppose that $(f_n)^\bullet\in I$.  Then by the definition of $I$, we have that $\Phi((f_n)^\bullet)=0$, so $\Phi$ induces a surjection $\Phi:A^{c\u}/I\to C_b(X)$.  Suppose now that $\Phi((f_n)^\bullet)=0$.  For each $n>0$, define a function $g_n\in A$ by $g_n=f_n\chi_{n-1}$.  It is clear that $(g_n)^\bullet\in I$.  Since $\lim_\u f_n(t)=0$ for all $t\in X$ and $\|f_n-g_n\|\leq\max_{t\in B(o,n)}|f_n(t)|$, it follows that $(f_n)^\bullet=(g_n)^\bullet$, whence $(f_n)^\bullet\in I$, thus proving that $\Phi:A^\u/I\to C_b(X)$ is an isomorphism.

Finally, it is clear from the definition of $\Phi$ that $\Phi(q(\Delta(a))=a$ for all $a\in A$.
\end{proof}

From now on, we set $B:=(A^{c\u}/I)/((q\circ \Delta)(A))$, a unital \cstar-algebra, and let $r:A^\u\to B$ denote the composition of $q$ with the quotient map $A^{c\u}/I\to B$.  Note that by the previous proposition, $B\cong C(X^*)$, where $X^*$ denotes the \emph{Stone-Cech remainder} $\beta X \setminus X$ of $X$.
 
We now introduce a group action into the picture:

\begin{lemma}
Suppose that $\Gamma$ acts isometrically on $X$. 
\begin{enumerate}
\item The induced action of $\Gamma$ on $A$ further induces an action of $\Gamma$ on $A^\u$ by $\gamma\cdot (f_n)^\bullet:=(\gamma\cdot f_n)^\bullet$.
\item Both $A^{c\u}$ and $I$ are invariant under the action of $\Gamma$ on $A^\u$ from (1).
\end{enumerate}
\end{lemma}

\begin{proof}
For (1), we need to verify that, for $(f_n),(g_n)\in \ell^\infty(A)$, if $\lim_\u \|f_n-g_n\|=0$, then $\lim_\u \|\gamma\cdot f_n-\gamma\cdot g_n\|=0$.  However, this follows from the easy check that $\|\gamma\cdot f_n-\gamma\cdot g_n\|=\|f_n-g_n\|$ for each $n$.

We now prove (2).  The fact that $A^{c\u}$ is invariant under the action of $\Gamma$ follows from the fact that $\Gamma$ acts by isometries and thus takes bounded sets to bounded sets.  We now prove that $I$ is invariant under the action of $\Gamma$.  Consider $\gamma\in \Gamma$ and $(f_n)^\bullet \in I$; we must show that $(\gamma \cdot f_n)^\bullet \in I$.  Suppose that $f_n|B(o,r_n)\equiv 0$ where $\lim_\u r_n=\infty$.  Set $k:=d(\gamma^{-1}\cdot o,o)$.  Then for $r_n>k$, we have that $(\gamma\cdot f_n)|B(o,r_n-k)\equiv 0$:  if $t\in B(o,r_n-k)$, then
$$d(\gamma^{-1}t,o)\leq d(\gamma^{-1}t,\gamma^{-1}o)+d(\gamma^{-1}o,o)=d(t,o)+k\leq r_n,$$
whence $f_n(\gamma^{-1}t)=0$.  Since $r_n>k$ for $\u$-most $n$ and $\lim_u r_n-k=\infty$, it follows that $(\gamma \cdot f_n)^\bullet\in I$.
\end{proof}

By the previous lemma, we have an induced action of $\Gamma$ on $A^{c\u}/I$ by setting $\gamma \cdot q(f_n)^\bullet:=q(\gamma\cdot f_n)^\bullet$, whence we also get an action of $\Gamma$ on $B$ by setting $\gamma \cdot r(f_n)^\bullet:=r(\gamma\cdot f_n)^\bullet$.

%
%

\section{Groups acting properly and isometrically on a tree}\label{proper}

In this section, our locally compact space is simply a tree $T$ given the usual path metric, namely $d(x,y)=$ the length of the shortest path connecting $x$ and $y$.  In this case, $A^{c\u}=A^\u$.  We further suppose that $\Gamma\curvearrowright T$ properly, isometrically and transitively.  (Recall that the action is proper if the map $(g,t)\mapsto (gt,t):G\times T\to T\times T$ is proper, meaning that inverse images of compact sets are compact.)  In this case, $\operatorname{Stab}(o)$ is finite, say of cardinality $m$.
For a point $t\in T$, we let $x_{[o,t]}$ denote the geodesic segment connecting $o$ and $t$.

\begin{thm}
If $\Gamma$ acts properly, transitively, and isometrically on a simplicial tree $T$, then $\Gamma$ is exact.
\end{thm}

\begin{proof}
For $t\in T$ and $i\in \n$, set $$X(i,t):=\{\gamma \in \Gamma \ : \ \gamma\cdot o\in B(i) \text{ and }\gamma\cdot o\in x_{[o,t]}\}$$ and $x(i,t)=|X(i,t)|^{-1/2}$.  Note that  $x(i,t)=m\cdot min(i,d(o,t))$.  Define $T_i^{(n)}:\Gamma\to A$ by
\[
T_i^{(n)}(\gamma)(t)=
\begin{cases}
x(i,t) 	&\text{if $t\in B(2n) \text{ and }\gamma \in X(i,t)$};\\
0 &\text{otherwise}.
\end{cases}
\]
Now define $T_i:\Gamma\to B$ by $T_i(\gamma):=r((T_i^{(n)}(\gamma)^\bullet)$.  We claim that these functions satisfy the criteria of Fact \ref{test}, whence the action of $\gamma$ on $X^*$ is amenable.

Certainly each $(T_i^{(n)}(\gamma))^\bullet$ is a positive element of $A^\u$; since $r$ is a \cstar-algebra homomorphism, we have that each $T_i(\gamma)\geq 0$ in $B$.

We now verify that $\langle T_i,T_i\rangle_B=1_B$; in other words, we must show that $\sum_{\gamma \in \Gamma}T_i(\gamma)^2=1_B$.  First observe that there is a finite $\Gamma_i\subseteq \Gamma$ such that $\sum_{\gamma \in \Gamma}T_i^{(n)}(\gamma)^2=\sum_{\gamma \in \Gamma_i}T_i^{(n)}(\gamma)^2=\chi_{B(2n)}$.  Since $(\chi_{B(2n)})^\bullet+I$ is the unit of $A^\u/I$, it follows that $r((\chi_{B(2n)})^\bullet)$ is the identity of $B$.  Now compute:
\begin{alignat}{2}
\sum_{\gamma \in \Gamma} T_i(\gamma)^2&=\sum_{\gamma \in \Gamma_i} T_i(\gamma)^2\notag \\ \notag
								&=\sum_{\gamma\in \Gamma_i}(r((T_i^{(n)}(\gamma))^\bullet)^2\notag \\
								&=\sum_{\gamma \in \Gamma_i}r(((T_i^{(n)}(\gamma))^\bullet)^2) \notag \\
								&=r(\sum_{\gamma \in \Gamma_i}((T_i^{(n)}(\gamma))^\bullet)^2))\notag \\
								&=r(\sum_{\gamma \in \Gamma_i}((T_i^{(n)}(\gamma))^2)^\bullet)) \notag \\
								&=r((\sum_{\gamma \in \Gamma_i}T_i^{(n)}(\gamma)^2)^\bullet)\notag \\
								&=r((\chi_{B(o,2n)})^\bullet)\notag \\
								&=1_B.\notag
\end{alignat}

It remains to prove that, for each $\gamma_1\in \Gamma$, we have $\lim_{i\to \infty} \|T_i-\delta_{\gamma_1}*T_i\|_2=0$.  It is straightforward to compute that $\delta_{\gamma_1}*T_i=\gamma_1\cdot T_i(\gamma_1^{-1}\gamma)$.  It follows that $\|T_i-\delta_{\gamma_1}*T_i\|_2^2$ is equal to
$$\|\sum_{\gamma \in \Gamma}(T_i(\gamma)^2+(\gamma_1\cdot T_i(\gamma_1^{-1}\gamma))^2-2T_i(\gamma)\gamma_1\cdot T_i(\gamma_1^{-1}\gamma))\|_B. \quad (\dagger)$$  Now $\gamma_1T_i^{(n)}(\gamma_1^{-1}\gamma)(t)=T_i^{(n)}(\gamma_1^{-1}\gamma)(\gamma_1^{-1}(t))$, which is only nonzero if:
\begin{enumerate}
\item $\gamma_1^{-1}t\in B(2n)$;
\item $\gamma_1^{-1}\gamma \cdot o\in B(i)$;
\item $\gamma_1^{-1}\gamma \cdot o\in x_{[o,\gamma_1^{-1}t]}$.
\end{enumerate}
Also notice that $\sum_{\gamma \in \Gamma}(\gamma_1\cdot T_i^{(n)}(\gamma_1^{-1}\gamma))^2=\chi_{\gamma_1\cdot B(2n)}$, so $(\dagger)$ equals
$$\|r(\chi_{B(2n)})^\bullet+r(\chi_{\gamma_1B(2n)})^\bullet -2\sum_{\gamma \in \Gamma}(T_i^{(n)}(\gamma)\gamma_1\cdot T_i^{(n)}(\gamma_1^{-1}\gamma))^\bullet)\|_B,$$ which in turn equals
$$\inf \{\lim_{n,\u}\|\chi_{B(2n)}+\chi_{\gamma_1\cdot B(2n)}-2\sum_{\gamma \in \Gamma}T_i^{(n)}(\gamma)\gamma_1\cdot T_i^{(n)}-g_n-a\|\}, \quad (\dagger \dagger)$$ where $(g_n)^\bullet$ ranges over $I$ and $a$ ranges over $A$.

Set
$$a(t)=(\chi_{B(2n)}+\chi_{\gamma_1\cdot B(2n)}-2\sum_{\gamma \in \Gamma}T_i^{(n)}(\gamma)\gamma_1\cdot T_i^{(n)})\cdot \chi_{B(i)\cup \gamma_1\cdot B(i)}.$$
Set $g_n(t)=\chi_{B(2n)\triangle \gamma_1\cdot B(2n)}$.  Finally set $$O(i,n)=(B(2n)\cap \gamma_1\cdot B(2n))\setminus (B(i)\cup \gamma_1\cdot B(i)).$$  Then $a\in A$, $(g_n)^\bullet\in I$ (as $g_n|B(n)\equiv 0$) and the value in $(\dagger \dagger)$ is bounded by
$$\lim_{n,\u}\sup_{t\in O(i,n)}|2-2\sum_{\gamma \in \Gamma}T_i^{(n)}(\gamma)\gamma_1\cdot T_i^{(n)})(t)|. \quad(\dagger \dagger \dagger)$$
Let $Z(i,t)$ be the set
$$\{\gamma \in \Gamma \ : \ \gamma\cdot o\in B(i), \gamma \cdot o\in x_{[o,t]}, \gamma_1^{-1}\gamma \cdot o\in B(i), \gamma_1^{-1}\gamma \cdot o\in x_{[o,\gamma_1^{-1}t]}\}.$$  Set $k:=d(\gamma_1\cdot o,o)$.  For $n$ sufficiently large and for $t\in O(i,n)$, we have $|Z(i,t)|=i-k$ and $2\sum_{\gamma\in \Gamma}(T_i^{(n)}(\gamma)\gamma_1\cdot T_i^{(n)})(t)=\frac{i-k}{im}$, whence the quantity appearing in $(\dagger \dagger \dagger)$ equals $2-2\frac{i-k}{im}$, which goes to $0$ as $i\to \infty$ as desired.
\end{proof}

\end{document}